\documentclass[a4paper,12pt]{article}

\usepackage{graphicx,amssymb}
\usepackage{textcomp}
\usepackage[bookmarks=true]{hyperref}
\newcommand{\bd}{\begin{displaymath}}
\newcommand{\be}{\begin{equation}}
\newcommand{\beq}{\begin{eqnarray}}
\newcommand{\ba}{\begin{array}}
\newcommand{\ed}{\end{displaymath}}
\newcommand{\ee}{\end{equation}}
\newcommand{\eeq}{\end{eqnarray}}
\newcommand{\ea}{\end{array}}
\newcommand{\espace}{\mbox{ }}

\newcommand{\N}{{\mathbb N}}

\newcommand{\Z}{{\mathbb Z}}

\newcommand{\eqref}[1]{(\ref{#1})}

\newtheorem{theorem}{Theorem}[section]

\newtheorem{proposition}{Proposition}[section]

\newtheorem{lemma}{Lemma}[section]
\newtheorem{corollary}{Corollary}[section]

\newenvironment{proof}[2]{\espace\\{\em Proof of #1 \ref{#2}.}}{\hfill\mbox{$\square$}}
\begin{document}

\title{Extreme paths in oriented 2D Percolation}
\author{E. D. Andjel$^{a,b}$ and L. F. Gray$^{c}$}

\maketitle

$$ \ba{l}
^a\,\mbox{\small Universit\'e d'Aix-Marseille,
 39 Rue Joliot Curie,
 13453 Marseille, France} \\
^b\, \mbox{\small Visiting IMPA, supported by CAPES and CNPq} \\
^c\, \mbox{\small School of Mathematics, University of Minnesota} 
\ea
$$

\begin{abstract}
A useful result about leftmost and rightmost paths in two dimensional bond percolation is proved.  This result was introduced without proof in \cite{G} in the context of the contact process in continuous time.  As discussed here, it also holds for several related models, including the  discrete time contact process and two dimensional site percolation.  Among the consequences are a natural monotonicity in the probability of percolation between different sites and a somewhat counter-intuitive correlation inequality.  
\end{abstract}


\section{Introduction}\label{sec:intro}

The main goal of this paper is to give complete proofs of a result originally presented in \cite{G}. The result was stated for the continuous time contact process in \cite{G}, but its proof is missing in the literature. In that paper some interesting consequences are given which  we believe justify the writing of the proof here. In the present paper we work in the context of oriented two dimensional percolation which is equivalent to a discrete time version of the contact process. In the latter part of this paper we discuss how our results apply to other models, and derive some consequences following the ideas of \cite{G}.

Two dimensional oriented bond percolation is studied in  \cite{D},  where some of its most important proprrties are proved. 
To introduce the model, let
$$\Lambda =\{(x,y): x, y \in \Z, y\geq 0, x+y \in 2\Z\} \, .$$
Then, draw oriented edges from each point $(m,n)$ in $\Lambda$ to $(m+1,n+1)$ and
to $(m-1,n+1)$ . In the percolation literature, the points in $\Lambda$ and the edges between them are often called sites and bonds, respectively.  In this paper we focus on oriented bond percolation, and thus we suppose that the edges
are open independently of each other, and that each edge is open with
probability $p\in (0,1)$.  

It is an easy matter to adapt the proof here to oriented site percolation, in which the points, rather than the edges, are open with probability $p$, independently of each other.  It turns out that our arguments also continue to work for a more general version of oriented bond percolation in which we allow dependence within each bond pair that emerges from a site.  And either by using the so-called ``graphical construction'' of continuous time interacting particle systems, or by taking limits of discrete time contact processes, our methods can also be applied to various versions of the continuous time contact process in one dimension. For more about such extensions, see our discussion in Section~\ref{sec5}.  

A path $\pi$ in $\Lambda $ is a sequence $(x_0,y_0),\dots,(x_n,y_n)$ of points in $\Lambda $ such that for all $0 \leq i< n$,  $\vert x_{i+1}-x_i\vert =1$ and $y_{i+1}-y_i=1$. The edges joining $(x_i,y_i)$ to $(x_{i+1},y_{i+1})$ for $0\leq i\leq n-1$ will be called the edges of
$\pi$. We say that a path is open if all its edges are open.  (In site percolation, a path is open if all its points are open.)

For any $n\in \N_0$ let $L_n=\{(x,n)\in \Lambda\}$. 
Let $0\leq m<n$ and let $A$ and $B$ be subsets of $L_m$ and $L_n$ respectively.
A path from $A$ to $B$ is any path starting in some point in $A$ and finishing at some point in $B$. 
A path $\pi$ from a point in $L_m$ to a point in $L_n$ will be identified with the function $\pi :[m,n] \cap \Z \rightarrow \Z$  determined by: $(\pi(j),j)$ is a point in the path $\pi$ for all $m\leq j \leq n$.


Given two paths $\pi_1$ and $\pi_2$ from $L_m$ to $L_n$ we say that $\pi_1$ is to the left of $\pi_2$ (or that $\pi_2$ is to the right of $\pi_1$) and write
$\pi_1\leq \pi_2$ (or $\pi_2\geq \pi_1$) if $\pi_1(j)\leq \pi_2(j)$  for all $m\leq j\leq n$. This creates a partial order
between paths from $L_m$ to $L_n$.  If the inequality is replaced by strict inequality, then we say that $\pi_1$ is strictly to the left of $\pi_2$ (or $\pi_2$ is strictly to the right of $\pi_1$).  

We find it useful to extend the notions ``strictly to the left'' and ``strictly to the right'' to subsets of $\Lambda$.
Let $P_1:\Lambda \rightarrow \Z$ be the projection on the first coordinate: $P_1((x,y))=x$. And in the usual fashion, extend this function to sets $\Lambda' \subset \Lambda$: $P(\Lambda') = \{P_1((x,y)) : (x,y) \in \Lambda'\}$.
For $G$  a subset of $\Lambda$, we denote by $\ell(G)$  the set of all points $(j,k) \in \Lambda$ such that
$j < \inf P_1(G \cap L_k)$, and we denote by $r(G)$ the set of points $(j,k) \in \Lambda$ such that $j > \sup P_1(G \cap L_k)$. (Here by convention, $\sup\emptyset=-\infty$ and $\inf \emptyset=\infty$.)   Thus, $\ell(G)$ ($r(G)$) is the set of all points in $\Lambda$ that are strictly to the left (right) of $G$.
If $G,G'$ are subsets of $\Lambda$, we say that $G$ is strictly to the left of $G'$ or, equivalently, $G'$ is strictly to the right of $G$, if $G \subset \ell(G')$, or equivalently if $G' \subset r(G)$.  
And if $\pi$ is a path, then we say that $\pi$ is strictly to the left of $G$ (strictly to the right of $G$) if the set of points in $\pi$ is strictly to the left of $G$ (strictly to the right of $G$).   The notation $\ell(\cdot)$ and $r(\cdot)$ introduced here also applies to paths, thinking of them as sets.  For example, a path $\pi$ is strictly to the left of a set $G$ if and only if  $G \subset r(\pi)$.  Please note that this terminology and notation are consistent with our earlier definition of one path being strictly to the left of another path, but that they now also apply to paths that do not necessarily start on the same level $L_m$ or end on the same level $L_n$.

Let $\Lambda' \subset \Lambda$.
Note that, if $A\subset L_m$ and $B\subset L_n$ are finite and there is at least one path from $A$ to $B$ contained in $\Lambda'$, then there is a unique path 
from $A$ to $B$ contained in $\Lambda'$ which is to the left of all paths from $A$ to $B$ contained in $\Lambda'$. This is called the leftmost path from $A$ to $B$.  And, if there is an open path from $A$ to $B$ contained
 in $\Lambda'$, then there is
a unique open path from $A$ to $B$ contained in $\Lambda'$ which is to the left of all open paths from $A$ to $B$ contained in $\Lambda'$. This path will be called the leftmost open 
path from $A$ to $B$ in $\Lambda'$. Similarly , we define the rightmost path and rightmost open path from $A$ to $B$ in $\Lambda'$.

Given a subset $\Lambda'$ of $\Lambda$, $0\leq m<n\in \N$ and two finite subsets $A$ and $B$ of  $L_m$ and $L_n$ respectively, $\Gamma_{\Lambda'}(A,B)$ will 
denote the set of paths from 
$A$ to $B$ contained in $\Lambda'$ . If this set is non-empty, then $\mu_{\Lambda'}(A,B)$ ($\nu_{\Lambda'}(A,B)$) will denote the conditional distribution of
 the leftmost (rightmost) open path from $A$ to $B$ contained in $\Lambda'$ given that there is at least one open path from $A$ to $B$ in $\Lambda'$.
If $m\leq j  \leq n$ and $C\subset L_j\cap \Lambda'$, $\Gamma_{\Lambda'}(A,C,B)$ will denote the set of paths from 
$A$ to $B$ going through a point in $C$ and  contained in $\Lambda'$, and if this set is non-empty, then $\mu_{\Lambda'}(A,C,B)$ ($\nu_{\Lambda'}(A,C,B)$) will
 denote the conditional distribution of the leftmost (rightmost) open path from $A$ to $B$ contained in $\Lambda'$ and going through a point
in $C$ given that there is at least one such open path.
In all these notations the subscript $\Lambda'$ will be omitted if $\Lambda'$ is the whole set $\Lambda$, and  when either  $A$,$B$
or $C$ is a singleton,
say $\{(x,y)\}$, we will often write $(x,y)$ rather than $\{(x,y)\}$. Finally, if $\gamma_1\in \Gamma (A,(x,y))$ and $\gamma_2\in \Gamma ((x,y),B)$ 
then $\gamma_1\gamma_2\in \Gamma(A,(x,y),B)$ will be the
concatenation of $\gamma_1$ and $\gamma_2$.  When the coordinates are not important, we will often denote a point $(x,y) \in \Lambda$ as a single boldface letter, such as $\mathbf z = (x,y)$, so that, for example, $\Gamma(A,(x,y),B)$ might be written as $\Gamma(A,\mathbf z, B)$.

Let $0\leq m <n$, let $A$ and  $B$ be subsets of $L_m$ and $L_n$ respectively such that $\Gamma(A,B)$ is nonempty, and  let $\mu$ and $\nu$ be probability measures on $\Gamma(A,B)$.   We say that $\mu$ is stochastically to the left
of $\nu$ and write $\mu\leq \nu$ if for any increasing function $\Phi$ on $\Gamma (A,B)$ we have \newline
$\int_{\Gamma (A,B)} \Phi(\gamma)d\mu (\gamma)\leq \int_{\Gamma (A,B)} \Phi(\gamma)d\nu (\gamma)$.

We can now state our version of the main result in \cite{G}:

\begin{theorem}\label{t1}
 Let $0\leq m < n$, let $A$ and $B$ be finite subsets of $L_m$ and $L_n$ respectively and let $G$ be a subset of $\Lambda$.
If $\Gamma_{\ell(G)}(A,B)$ is nonempty, we have
$\mu_{\ell(G)}(A,B)\leq \mu(A,B)$ and $\nu_{\ell(G)}(A,B)\leq \nu(A,B)$. And if $\Gamma_{r(G)}(A,B)$ is nonempty, 
we have $\mu_{r(G)}(A,B)\geq \mu(A,B)$ and $\nu_{r(G)}(A,B)\geq \nu(A,B)$.
\end{theorem}

This result has the following immediate corollary:
\begin{corollary}\label{t2}

  Let $m<n$, let $A$ and $B$ be finite subsets of $L_m$ and $L_n$ respectively such that $\Gamma(A,B)$ is nonempty and suppose $\mathbf a$ is a point in $L_m$ that is strictly to the right of $A$.
Then, $\mu(A\cup \{\mathbf a\},B)\geq  \mu(A,B)$ and $\nu(A\cup \{\mathbf a\},B)\geq  \nu(A,B)$. Moreover, if $\mathbf b$ is a point in $L_n$ that is strictly to the right of $B$, then 
$\mu(A ,B\cup \{\mathbf b\})\geq  \mu(A,B)$ and  $\nu(A ,B\cup \{\mathbf b\})\geq  \nu(A,B)$. If, instead $\mathbf a$ is strictly to the left of $A$  (\ $\mathbf b$ is strictly to the left of $B$) then the first two (last two) inequalities are reversed.
\end{corollary}

The different parts of the corollary follow from the theorem by making appropriate choices of the sets $A,B,G$. For example, for the first part of the corollary, replace $A$ in the theorem by $A \cup \{\mathbf a\}$ and let $G = \{\mathbf a\}$.   

In Section  2 we  prove two elementary lemmas. Then, in Section 3 we prove a key proposition 
and with it in hand we prove the
theorem, using a somewhat involved inductive argument that was hinted at in \cite{G}. In Section 4, we give an alternate proof of our main theorem, based on a Markov chain that was introduced in \cite{BHK}.  We discovered this approach after we had fully developed our inductive argument. The Markov chain argument is shorter, but we have not been able to generalize it to other models quite as well as our inductive argument.  Finally, in the last two sections, we discuss extensions to other models and derive some consequences of the theorem and its corollary following the ideas of \cite{G}.

\section{Basic lemmas}\label{sec:bl}
We start this section with a very simple lemma:

\begin{lemma}  \label{l1} Let $m<n$ and let $G$ be a non-empty subset of $\Lambda$ such that the projection $P_1(G)$ is bounded below. Then, there exists a path $\tau_G$  
from $L_m$ to $L_n$ such that any other path $\gamma$ from $L_m$ to $L_n$  is strictly to the left  of $G$ if and only if it is strictly to the left of $\tau_G$. If instead $P_1(G)$ is bounded above, the same statement holds if we substitute right for left.
\end{lemma}
\begin{proof}{Lemma}{l1} Suppose $P_1(G)$ is bounded below. 
For each $(x,y)\in G$, let $B_{x,y}= \{(u,v)\in \Lambda : u\geq x+\vert v-y\vert \}$ and let
$B(G)=\cup_{(x,y)\in G}B_{x,y}$. Then  for $m\leq j \leq n$ define $x_j=\inf P_1(B(G)\cap L_j)$. It is now easy to verify that the sequence $(x_j,j): m\leq j \leq n$ defines
 a path $\tau_G$ from $L_m$ to $L_n$
satisfying the conclusion of the lemma. A similar proof works when $P_1(G)$ is bounded above.
\end{proof}

\bigskip

Before stating our next lemma, we  introduce some further notation: Fix integers $0 \leq m < j < n$ and sets $A \subset L_m, B \subset L_n, C \subset L_j$.  Given a probability measure $\rho $ on 
 $\Gamma(A,C,B) $ we call $\rho_1$ and $\rho_2$ its marginals on $\Gamma(A,C) $ and  $\Gamma(C,B) $ 
respectively. They are given by:
$$\rho_1(\gamma_1) =\sum_{\gamma_2 \in \Gamma (C,B)}\rho(\gamma_1\gamma_2)\mbox{ and}$$
$$\rho_2(\gamma_2) =\sum_{\gamma_1 \in \Gamma (A,C)}\rho(\gamma_1\gamma_2).$$
Given $\gamma_1 \in \Gamma (A,C)$ such that $\rho_1(\gamma_1)>0$, we define the conditional measure  $\rho( \bullet \vert \gamma_1)$ on $\Gamma(C,B) $  by:
$$\rho(\gamma_2\vert \gamma_1)=\frac{\rho (\gamma_1\gamma_2)}{\rho_1(\gamma_1)}.$$  We note that we have perhaps abused the conditional probability notation here slightly, since it may seem more technically accurate to let $\rho(\bullet \vert \gamma_1)$ denote a probability
measure on paths from $A$ to $B$ (with the portion from $A$ to $C$ equaling $\gamma_1$), rather than the way we have defined it, which is as a probability measure on paths from $C$ to $B$, but we hope that this abuse will not cause any confusion for the reader.

\bigskip

The following lemma is needed to prove the key proposition stated in the next section.

\begin{lemma}\label{l2}
Let $m, j$ and $n$ be integers such that $0\leq m<j<n$. Then, let $A\subset L_m$ and $B\subset L_n$ be finite and let ${\bf a} \in L_j$ be such that $\Gamma(A,{\bf a})$
and $\Gamma({\bf a},B)$ are nonempty. Finally, let  $\rho_1$ and $\rho_2$ be probability measures on
 $\Gamma (A,{\bf a})$ and on $\Gamma({\bf a},B)$ respectively and let $\rho$ be the probability measure on 
 $\Gamma(A,{\bf a},B) $ defined by $\rho(\gamma_1\gamma_2)=\rho_1(\gamma_1)\rho_2(\gamma_2)$.

If $\sigma $ is a probability measure on  $\Gamma(A,{\bf a},B) $ such that

i) $\rho_1\leq \sigma_1$.

ii)  $\rho_2\leq \sigma(\bullet \vert \gamma_1)$  for any $\gamma_1 \in \Gamma (A,{\bf a})$ such that $\sigma_1(\gamma_1)>0$,
\newline then $\rho\leq \sigma$.

\end{lemma}

\begin{proof}{Lemma}{l2}
 Let $\phi$ be increasing on $\Gamma(A,{\bf a},B) $. Then write
$$\int_{\Gamma(A,{\bf a},B) } \phi(\gamma)d\sigma(\gamma)=$$
$$\sum_{\gamma_1\in \Gamma(A,{\bf a}) }\Big( \sum_{\gamma_2 \in \Gamma({\bf a},B) }\phi(\gamma_1 \gamma_2)\sigma(\gamma_2\vert \gamma_1)\Big) \sigma_1(\gamma_1)\geq$$
$$\sum_{\gamma_1\in \Gamma(A,{\bf a})}\Big( \sum_{\gamma_2\in \Gamma({\bf a},B)}\phi(\gamma_1 \gamma_2)\rho_2(\gamma_2 )\Big) \sigma_1(\gamma_1)\geq$$
$$\sum_{\gamma_1\in \Gamma(A,{\bf a})}\Big( \sum_{\gamma_2\in \Gamma({\bf a} ,B)}\phi(\gamma_1 \gamma_2)\rho_2(\gamma_2 )\Big) \rho_1(\gamma_1)=$$
$$\int_{\Gamma(A,{\bf a},B) } \phi(\pi) d\rho(\pi),$$
where the first inequality follows from the fact that for all $\gamma_1 \in  \Gamma(A,{\bf a}) $ $\sigma(\bullet \vert \gamma_1)\geq \rho_2$ and 
$\gamma_2\rightarrow \phi(\gamma_1 \gamma_2)$ is an increasing function and the second inequality follows from the fact that 
 $\sigma_1\geq \rho_1$ and $$\gamma_1\rightarrow \sum_{\gamma_2\in \Gamma({\bf a},B)} \phi(\gamma_1 \gamma_2)\rho_2(\gamma_2)$$ is an increasing function.\end{proof}

\section{Proof of Theorem \ref{t1}}\label{proofs}

We have defined $\mu(A,B)$ as the conditional distribution of the leftmost open path from $A$ to $B$ given the event
$$H=\{\mbox{there exists an open path from } A \mbox{ to } B\}.$$
In this section we adopt the following notation: if $F$ is another event such that $P(F\cap H)>0$, then
$\mu^F (A,B)$  is the distribution of the leftmost  open path from $A$ to $B$ given the event $F\cap H$. The same notation will apply to distributions such as
$\mu(A,C,B)$, $\mu_{\Lambda_0}(A,B)$ etc.\ and to distributions of rightmost open paths such as $\nu(A,B)$, $\nu(A,C,B)$, $\nu_{\Lambda_0}(A,B)$ etc.

We now extend the notion of paths from $L_m$ to $L_n$ by adding two extra paths: the sequence $(-\infty,m),\dots,(-\infty,n)$ and  the sequence $(\infty,m),\dots,(\infty,n)$. We often suppress $m$ and $n$ and simply denote these paths
by $-\infty$ and $\infty$ respectively. The path $-\infty$ ($\infty$) will be considered as being strictly to the left (right)
of any other path from $L_m$ to $L_n$, and also of any subset of $\Lambda$. Finally, given $0\leq m<n$ and two paths $\tau_1$ and $\tau_2$ from 
$L_m$ to $L_n$ such that $\tau_1$ is strictly to the left of $\tau_2$ we let $b(\tau_1,\tau_2)$ be the set of points that are strictly to the right of $\tau_1$ and strictly to the left of $\tau_2$.  

We now state a proposition which is a slightly weaker version of Theorem \ref{t1}, namely:
\begin{proposition}\label{p1}
 Let $0\leq m<n$, let $\tau_1$,$\tau_2$ and $\tau_3$ be  paths from  $L_m$ to  $L_n$ such that $\tau_1$ is  to
the left of $\tau_3$ and $\tau_3$ is  to the left of  $\tau_2$,
and let $A$ and $B$ be non-empty finite subsets of $L_m$ and $L_n$ respectively.
If $A$ is strictly to the right of $\tau_3$ and there exists a path from $A$ to $B$ in $b(\tau_3,\tau_2)$,  then
  $\mu_{b(\tau_1,\tau_2)}(A,B)\leq \mu_{b(\tau_3,\tau_2)}(A,B)$ and
$\nu_{b(\tau_1,\tau_2)}(A,B)\leq \nu_{b(\tau_3,\tau_2)}(A,B)$. Similarly, if
$A$ is strictly to the left of $\tau_3$ and there exists a path from $A$ to $B$ in $b(\tau_1,\tau_3)$ we have  $\mu_{b(\tau_1,\tau_3)}(A,B)\leq 
\mu_{b(\tau_1,\tau_2)}(A,B)$ and $\nu_{b(\tau_1,\tau_3)}(A,B)\leq 
\nu_{b(\tau_1,\tau_2)}((A,B)$.
\end{proposition}
The proof of this proposition requires a rather involved inductive argument. To facilitate its reading, we start deriving some consequences  of the inductive hypothesis 
of that proof. That is the purpose of our next lemma and its corollaries.
\begin{lemma}\label{l31}
 Suppose that Proposition \ref{p1} holds for $m=0$ and some $n>0$. Let $\tau_1$ and $\tau_2$ be paths from $L_0$ to $L_n$ such that $\tau_1$ is strictly to the 
left of $\tau_2$, let  $A$ and $B$ be finite subsets of $L_0$ and $L_n$ respectively,  and  let   $\mathbf a$ be a point in $L_0$ and $\mathbf b$ a point in $L_n$.
Assuming there exists a path in $b(\tau_1,\tau_2)$ from $A$ to $B$, the following statements hold:

i) If $\mathbf a$ is strictly to the right of $A$ then,
$$\mu_{b(\tau_1,\tau_2)}(A,B)\leq \mu_{b(\tau_1,\tau_2)}(A\cup \{\mathbf a\},B)\mbox{ and}$$
$$\nu_{b(\tau_1,\tau_2)}(A,B)\leq \nu_{b(\tau_1,\tau_2)}(A\cup \{\mathbf a\},B).$$

ii) If $\mathbf b$ is strictly to the right of $B$ then,
$$\mu_{b(\tau_1,\tau_2)}(A,B)\leq \mu_{b(\tau_1,\tau_2)}(A,B\cup \{\mathbf b\})\mbox{ and}$$
$$\nu_{b(\tau_1,\tau_2)}(A,B)\leq \nu_{b(\tau_1,\tau_2)}(A,B\cup \{\mathbf b\}).$$
If instead,
$\mathbf a$ is strictly to the left of $A$ the first two inequalities   are reversed while if
$\mathbf b$ is strictly to the left of $B$ the last two  are reversed. 
\end{lemma}

\begin{proof}{Lemma}{l31}
We only prove the first inequality since the proofs of the others are similar.  We may assume that there exists a path from $\mathbf a$ to $B$ between $\tau_1$ and $\tau_2$, since otherwise the result is trivial.
Let $$F_A=\{\mbox{ there is no open path in }b(\tau_1,\tau_2) \mbox{ from }A\mbox{ to } B\},$$
  let
$$F_{\mathbf a}=\{\mbox{ there is no open path in }b(\tau_1,\tau_2) \mbox{ from }\mathbf a\mbox{ to } B\}, $$
and let $\Phi$ be the random set of edges belonging to open paths starting  from $A$ and contained in $b(\tau_1,\tau_2)$ .
Then the event $F_A$ is the union of disjoint events of the form
$\{\Phi=\varphi\}$ where $\varphi$ ranges over a collection $\Upsilon$ of deterministic sets of edges,
and $\mu^{F_A}_{b(\tau_1,\tau_2)}(\mathbf a,B)$ is a convex combination of measures of the form $\mu^{\{\Phi=\varphi\}}_{b(\tau_1,\tau_2)}(\mathbf a,B)$ where 
$\varphi $ is such that there exists a path from $\mathbf a$ to $B$ in $b(\tau_1,\tau_2)$ which is strictly to the right of $\varphi$. Now, we let 
$V(\varphi)$ be the set of points which are vertices of edges in $\varphi$ and observe that on the event $\{\Phi=\varphi\}$ any open path from
$\mathbf a$ to $B$ must be strictly to the right of
$V(\varphi)$. By Lemma \ref{l1} there 
exists a path $\tau(\varphi)$ such that the paths which are  strictly to the right of $V(\varphi)$ are exactly those which are strictly to the right of $\tau(\varphi)$.  
It then follows that 
\begin{equation}\label{e31}
 \mu^{\{\Phi=\varphi\}}_{b(\tau_1,\tau_2)}(\mathbf a,B)=\mu _{b(\tau_1(\varphi),\tau_2)}(\mathbf a,B),
\end{equation}
 where $\tau_1(\varphi)$ is the 
leftmost path such that $\tau_1 \leq \tau_1(\varphi)$ and $\tau(\varphi)\leq \tau_1(\varphi)$.  We have used the fact that $\mathbf a$ is necessarily strictly between $\tau_1(\varphi)$ and $\tau_2$. We also needed the following:  given $\{\Phi = \varphi\}$, the conditional distribution of the openness of the bonds emanating from sites that are strictly to the right of $\tau_1(\varphi)$ is the same as the unconditional distribution. This fact follows easily from the independence that is built into the model. 

Since we are assuming that the conclusion
of Proposition \ref{p1} holds we obtain:
$$ \mu _{b(\tau_1,\tau_2)}(\mathbf a,B)\leq \mu^{\{\Phi=\varphi\}}_{b(\tau_1,\tau_2)}(\mathbf a,B).$$
Therefore,
\begin{equation}\label{e32}
 \mu _{b(\tau_1,\tau_2)}(\mathbf a,B)\leq \mu^{F_A}_{b(\tau_1,\tau_2)}(\mathbf a,B).
\end{equation}
Since $\mu _{b(\tau_1,\tau_2)}(\mathbf a,B)$ is a convex combination of 
$\mu^{F_A}_{b(\tau_1,\tau_2)}(\mathbf a,B)$ and of $\mu^{F_A^c}_{b(\tau_1,\tau_2)}(\mathbf a,B)$, \eqref{e32} implies:
\begin{equation}\label{e33}
\mu^{F_A^c}_{b(\tau_1,\tau_2)}(\mathbf a,B)\leq \mu _{b(\tau_1,\tau_2)}(\mathbf a,B).
\end{equation}
Similarly one shows that
\begin{equation}\label{e34}
\mu^{F_{\mathbf a}}_{b(\tau_1,\tau_2)}(A,B)\leq \mu_{b(\tau_1,\tau_2)}(A,B).
\end{equation}
and that
\begin{equation}\label{e35}
\mu_{b(\tau_1,\tau_2)}(A,B)\leq \mu^{F_{\mathbf a}^c}_{b(\tau_1,\tau_2)}(A,B).
\end{equation}
But on the event $F_A^c\cap F_{\mathbf a}^c$ the leftmost path from $A$ to $B$ is      to the left of the leftmost path from
$\mathbf a$ to $B$, hence
\begin{equation}\label{e36}
\mu^{F_{\mathbf a}^c}_{b(\tau_1,\tau_2)}(A,B)=
\mu^{F_A^c\cap F_{\mathbf a}^c}_{b(\tau_1,\tau_2)}(A,B)\leq
\mu^{F_A^c\cap F_{\mathbf a}^c}_{b(\tau_1,\tau_2)}(\mathbf a,B)= \mu^{F_A^c}_{b(\tau_1,\tau_2)}(\mathbf a,B).
\end{equation}
It now follows from \eqref{e33},\eqref{e35} and \eqref{e36} that:
\begin{equation}\label{e37}
\mu_{b(\tau_1,\tau_2)}(A,B)\leq \mu _{b(\tau_1,\tau_2)}(\mathbf a,B).
\end{equation}
Since on the event $F_A^c$ the leftmost path from $A\cup\{\mathbf a\}$ to $B$ is the same as the leftmost path from $A$ to $B$ we have:
\begin{equation}\label{e38}
 \mu^{F_A^c}_{b(\tau_1,\tau_2)}(A\cup \{\mathbf a\},B)=\mu_{b(\tau_1,\tau_2)}(A,B).
\end{equation}
And, since on the event  $F_A$ the leftmost path from $A\cup\{\mathbf a\}$ to $B$ is the same as the leftmost path from 
$\mathbf a$ to $B$ we also have:
\begin{equation}\label{e39}
 \mu^{F_A}_{b(\tau_1,\tau_2)}(A\cup \{\mathbf a\},B)=\mu^{F_A}_{b(\tau_1,\tau_2)}(\mathbf  a,B).
\end{equation}
  Therefore $ \mu_{b(\tau_1,\tau_2)}(A\cup \{\mathbf a\},B)$ is a convex combination of
$\mu_{b(\tau_1,\tau_2)}(A,B)$ and $\mu^{F_A}_{b(\tau_1,\tau_2)}(\mathbf  a,B)$.
But, it follows from \eqref{e32} and \eqref{e37} that
$$\mu_{b(\tau_1,\tau_2)}(A,B)\leq \mu^{F_A}_{b(\tau_1,\tau_2)}(\mathbf a,B).$$
Therefore
$$\mu_{b(\tau_1,\tau_2)}(A,B)\leq \mu_{b(\tau_1,\tau_2)}(A\cup \{\mathbf a\},B).$$
 
\end{proof}

\begin{corollary}\label{c3}
 Assume the hypothesis of Lemma \ref{l31}  and let   $A'$ be  a finite subset of $L_0$. If $A'$ is strictly to the right of $A$, then
$$\mu_{b(\tau_1,\tau_2)}(A,B)\leq \mu_{b(\tau_1,\tau_2)}(A\cup A',B).$$
If in addition there exists a path in $b(\tau_1,\tau_2)$ from $A'$ to $B$ then,
$$\mu_{b(\tau_1,\tau_2)}(A,B)\leq \mu_{b(\tau_1,\tau_2)}( A',B).$$
If  $A'$ is strictly to the left of $A$, then
$$\mu_{b(\tau_1,\tau_2)}(A\cup A',B)\leq \mu_{b(\tau_1,\tau_2)}(A,B).$$
If in addition there exists a path in $b(\tau_1,\tau_2)$ from $A'$ to $B$ then
$$\mu_{b(\tau_1,\tau_2)}( A',B)\leq \mu_{b(\tau_1,\tau_2)}(A,B).$$
\end{corollary}
\begin{proof}{Corollary}{c3}
Assume $A'$ is strictly to the right of $A$. Then, let $\mathbf a'_1<\mathbf a'_2,\dots ,\mathbf a'_s$ be the points of $A'$ and let $A'_j=\{\mathbf a'_1,\dots,\mathbf a'_j\}$. 
It then follows from Lemma \ref{l31} that
$$\mu_{b(\tau_1,\tau_2)}(A,B)\leq \mu_{b(\tau_1,\tau_2)}(A\cup A'_1,B)\leq \dots \leq \mu_{b(\tau_1,\tau_2)}(A\cup A'_s,B),$$
and the first inequality is proved. The third inequality is proved in the same way. Then using the third inequality interchanging the roles of $A$ and $A'$ we get
$$\mu_{b(\tau_1,\tau_2)}(A\cup A',B)\leq \mu_{b(\tau_1,\tau_2)}(A',B),$$
and together with the first inequality, this  implies the second inequality. The fourth inequality is proved following the same method.
×
\end{proof}

In the same way we also get the following two corollaries:
\begin{corollary}\label{c4}
 Assume the hypothesis of Lemma \ref{l31} and  let $B'$ be a finite subset of $L_n$. If $B'$ is strictly to the right of $B$, then
$$\mu_{b(\tau_1,\tau_2)}(A,B)\leq \mu_{b(\tau_1,\tau_2)}(A,B\cup B').$$
If in addition there exists a path in $b(\tau_1,\tau_2)$ from $A'$ to $B$ then,
$$\mu_{b(\tau_1,\tau_2)}(A,B)\leq \mu_{b(\tau_1,\tau_2)}( A,B').$$
If  $B'$ is strictly to the left of $B$, then
$$\mu_{b(\tau_1,\tau_2)}(A,B\cup B')\leq \mu_{b(\tau_1,\tau_2)}(A,B).$$
If in addition there exists a path in $b(\tau_1,\tau_2)$ from $A$ to $B'$ then
$$\mu_{b(\tau_1,\tau_2)}( A,B')\leq \mu_{b(\tau_1,\tau_2)}(A,B).$$
\end{corollary}
 \begin{corollary}\label{c5}
  Corollaries \ref{c3} and \ref{c4} also hold if in their statements we substitute distributions of rightmost open paths 
for distributions of leftmost open paths (i.e. the inequalities also hold if we write $\nu$ instead of $\mu$). 
 \end{corollary}

\begin{proof}{Proposition}{p1}.
 Without loss of generality, we  assume that $m=0$ and  
proceed by
 induction on $n$, the inductive statement being that the four inequalities 
 hold under their respective hypotheses. 
For $n=1$ the result is trivial. For the inductive step, we wish to prove that the four inequalities hold under their respective hypotheses for a specific value of $n \geq 2$, given the inductive hypothesis, which is that the four inequalities all hold under their respective hypotheses for all smaller values of $n \geq 1$.   

We will prove that 
\begin{equation}\label{e0}
 \mu_{b(\tau_1,\tau_3)}(A,B)\leq \mu_{b(\tau_1,\tau_2)}(A,B)
\end{equation}
where $A$ and $B$ are subsets of $L_0$ and $L_n$ respectively and $\tau_1,\tau_2,\tau_3$ are paths from $L_0$ to $L_n$ satisfying the appropriate hypotheses, including the hypothesis that $A$ is strictly to the left of $\tau_3$. 
We omit the proofs of the other inequalities because
they follow the same ideas. Let  $\mathbf a_0,\dots,\mathbf a_n$  be the successive vertices
of $\tau_3$ and for $i=0,\dots,n$, let
\[
\Lambda_i = \ell(\{\mathbf a_0,\dots,\mathbf a_i\}) \cap b(\tau_1,\tau_2) \, .
\]
This is the set of vertices that are strictly between $\tau_1$ and $\tau_2$ and also strictly to the left of the first $i+1$ vertices of $\tau_3$.  

Of course, $\Lambda_n = b(\tau_1,\tau_3)$, so $\mu_{b(\tau_1,\tau_3)}(A,B) = \mu_{\Lambda_n}(A,B)$. 
Since we have assumed that $A$ is strictly to the left of $\tau_3$, we also have $\mu_{\Lambda_0}(A,B)=\mu_{b(\tau_1,\tau_2)}(A,B)$.  Therefore, the result will follow from:
\begin{equation}\label{e1}
 \mu_{\Lambda_{i+1}}(A,B)\leq \mu_{\Lambda_{i}}(A,B)\ \ i=0,\dots,n-1.
\end{equation}
We fix an $i\in \{0,\dots, n-1\}$ and note that $\mu_{\Lambda_{i+1}}(A,B)= \mu_{\Lambda_{i}}(A,B)$ if the paths from $A$ to $B$ in $\Lambda_{i+1}$ and in $\Lambda_{i}$ are the same. These sets of paths are the same if there is no path from $A$ to $B$ that goes through $\mathbf a_{i+1}$, or if $\mathbf a_{i+1}= \mathbf a_i + (1,1)$.  Thus, in proving \eqref{e1}, we may assume that there is at least one path from $A$ to $B$ going through $\mathbf a_{i+1}$ and that
$\mathbf a_{i+1} = \mathbf a_i + (-1,1)$.

We start proving \eqref{e1} for $i\in \{0,\dots, n-2\}$. Later, we will show it also holds for $i=n-1$.
Let 
\begin{equation}
F=\{\mbox{there is no open path from } A \mbox{ to } B \mbox{ in } \Lambda_{i+1}\}.
\end{equation}
We now show that 
\begin{equation}\label{e2}
 \mu_{\Lambda_{i}}(A,\mathbf a_{i+1},B)\leq \mu^F_ {\Lambda_{i}}(A,\mathbf a_{i+1},B) \, .
\end{equation}
In words, this says that  the leftmost open path from $A$ to $B$ contained in $\Lambda_{i}$ and passing through $\mathbf a_{i+1}$ is stochastically to the left of the leftmost open path 
from $A$ to $B$ contained in $\Lambda_{i}$ and passing through $\mathbf a_{i+1}$ given that there is no open  path from $A$ to $B$ in $\Lambda_{i+1}$.
To prove \eqref{e2}  we will show that  $\mu_{\Lambda_{i}}(A,\mathbf a_{i+1},B)$ and  $\mu_{\Lambda_{i}}^F(A,\mathbf a_{i+1},B)$ satisfy the 
hypothesis of Lemma \ref{l2}.
Clearly 
$$\mu_{\Lambda_{i}}(A,\mathbf a_{i+1},B)=\mu_{\Lambda_{i}}(A,\mathbf a_{i+1})\times \mu_{\Lambda_{i}}(\mathbf a_{i+1},B)$$
Hence, to show \eqref{e2} it suffices to show that the measures of this inequality satisfy the hypotheses i) and ii) of Lemma \ref{l2} . 

Let $\Phi$ be the random set of edges
 that belong to open paths in $\Lambda_i$ starting from points in  $A$ which   do not go through $\mathbf a_{i+1}$ (but may have $\mathbf a_{i+1}$ as their endpoint). 
Then the event $F$ can be expressed as the union of disjoint events of the form
$\{\Phi=\varphi\}$ where $\varphi$ ranges over some deterministic subsets of edges. We call $\Upsilon $ the collection of those subsets:
$F=\cup_{\varphi \in \Upsilon}\{\Phi=\varphi \}$. Fix a path $\gamma_1$ from $A$ to $\mathbf a_{i+1}$ in $\Lambda_i$ and  let $\Upsilon_{\gamma_1}$ be the collection of sets in 
$\Upsilon$ which contain all the edges of  $\gamma_1$ and contain no other path from $A$ to $\mathbf a_{i+1}$ in $\Lambda_i$ to the left of $\gamma_1$. Note that for any $\varphi \in \Upsilon_{\gamma_1}$, on the event $\{\Phi=\varphi\}$ any open path from
$\mathbf a_{i+1}$ to $B$ must be strictly to the right of $\varphi$.
Now, $ \mu_{\Lambda_i}^F(A,\mathbf a_{i+1},B)(\bullet \vert \gamma_1)$ is a convex combination of measures of the form 
$ \mu_{\Lambda_i}^{\{\Phi=\varphi\}}(A,\mathbf a_{i+1},B)_2 $ where $\varphi$ 
ranges over $\Upsilon_{\gamma_1}$.
But for each $\varphi$, the corresponding measure is equal to $\mu_{{\Lambda_i}\cap r(\varphi)}(\mathbf a_{i+1},B)$ i.e. the distribution of the leftmost open path from $\mathbf a_{i+1}$ to $B$ in $\Lambda_i$
 which is strictly to the right of $\varphi$.  In making this assertion, we are using the fact that given $\{\Phi = \varphi \}$, the conditional distribution of the openness of the bonds that emerge from sites that are strictly to the right of $\varphi$ is the same as the unconditional distribution.  This consequence of the independence in the model is similar to the one we used in the proof of Lemme~\ref{l31} when we proved \eqref{e31}. We note for future reference (see Section~\ref{sec5}) that both here and in the proof of Lemma~\ref{l31}, we do not need independence for two bonds that emerge from the same site; we only need it for bonds that emerge from different sites.

 By the inductive hypothesis, and Lemma \ref{l1} we have
 $$\mu_{{\Lambda_i}\cap r(\varphi)}(\mathbf a_{i+1},B)\geq \mu_{\Lambda_i}(\mathbf a_{i+1},B),$$
 for any $\varphi \in \Upsilon_{\gamma_1}$. 
Hence 
$$ \mu_{\Lambda_i}^F(A,\mathbf a_{i+1},B)(\bullet \vert \gamma_1)\geq \mu_{\Lambda_i}(\mathbf a_{i+1},B)=$$
$$\mu_{\Lambda_i}(A,\mathbf a_{i+1},B)_2.$$
Since this holds for all $\gamma_1$, it implies that 
$$\mu_{\Lambda_i}^F(A,\mathbf a_{i+1},B)_2 \geq \mu_{\Lambda_i}(A,\mathbf a_{i+1},B)_2.$$
Since $i\leq n-2$ we can use the inductive hypothesis in the same way 
to show that 
\begin{equation} \label{e2prime} \mu_{\Lambda_i}^F(A,\mathbf a_{i+1},B)_1 \geq \mu_{\Lambda_i}(A,\mathbf a_{i+1},B)_1.
\end{equation}
The one important difference in this argument is that we must partition the event $F$ according to open paths that end at the set $B$ instead of starting at $A$.  Now \eqref{e2} follows from Lemma \ref{l2}. 
Since $$\mu_{ \Lambda_i}^{F}(A,B)=\mu_{\Lambda_i}^F(A,\mathbf a_{i+1},B),$$

from \eqref{e2} we get:

$$ \mu_{ \Lambda_i}^{F}(A,B)\geq    \mu_{\Lambda_i}(A,\mathbf a_{i+1},B).$$
And since
$$\mu_{ \Lambda_i}^{F^c}(A,B)=\mu_{ \Lambda_{i+1}}(A,B),$$
\eqref{e1} will follow if we show that:
\begin{equation}\label{e2-1}
 \mu_{\Lambda_i}(A,\mathbf a_{i+1},B)\geq \mu_{ \Lambda_{i+1}}(A,B).
\end{equation}

To prove this, we let 
$$C_j=\{\mathbf a_{i+1}-(2j,0) ,\dots, \mathbf a_{i+1}-(2,0)\},$$
where $j$ ranges from $\underline j $ to $\bar j$,  where  $\underline j $ is the lowest value of $j$ for which $\Gamma_{\Lambda_{i+1}}(A,C_{\underline j},B)$ is nonempty and $\bar j$ is
the largest value of $j$ for which $C_j$ is strictly to the right of $\tau_1$. 
We will now prove by induction on $j$ that 
\begin{equation}\label{e2-2}
 \mu_{\Lambda_i}(A,\mathbf a_{i+1},B)\geq \mu_{ \Lambda_{i+1}}(A,C_{j},B),
\end{equation}
for all $\underline j \leq j \leq \bar j$.
Since
$\mu_{\Lambda_{i+1}}(A,B)=\mu_{ \Lambda_{i+1}}(A,C_{\bar j},B)$, \eqref{e2-1} will follow. 
For $j=\underline j$, first note that
$$\mu_{ \Lambda_{i+1}}(A,C_{\underline j},B)=\mu_{ \Lambda_{i+1}}(A,\mathbf a_{i+1} - (2 \underline j ,0))\times \mu_{ \Lambda_{i+1}}(\mathbf a_{i+1} - (2\underline j,0),B)$$
and 
$$\mu_{ \Lambda_{i}}(A,\mathbf a_{i+1},B)=\mu_{ \Lambda_{i}}(A,\mathbf a_{i+1})\times 
\mu_{ \Lambda_{i}}(\mathbf a_{i+1},B).$$
Since
$$\mu_{ \Lambda_{i+1}}(A,\mathbf a_{i+1} - (2 \underline j ,0))= \mu_{ \Lambda_{i}}(A, \mathbf a_{i+1} - (2  \underline j ,0)),$$
from Corollary \ref{c4} and the inductive hypothesis on $n$ we get:
$$\mu_{ \Lambda_{i+1}}(A,\mathbf a_{i+1} - (2 \underline j ,0)))\leq \mu_{ \Lambda_{i}}(A,\mathbf a_{i+1}).$$
Similarly, since
$$\mu_{ \Lambda_{i+1}}(\mathbf a_{i+1} - (2  \underline j ,0),B)= \mu_{ \Lambda_{i}}( \mathbf a_{i+1} - (2 \underline j ,0),B),$$
from Corollary \ref{c3} and the same inductive hypothesis we get: 
$$\mu_{ \Lambda_{i+1}}(\mathbf a_{i+1} - (2  \underline j ,0),B)\leq \mu_{ \Lambda_{i}}(\mathbf a_{i+1},B).$$
Therefore for $j= \underline j $, \eqref{e2-2} follows from  Lemma \ref{l2}.
To prove the inductive step, we  assume that \eqref{e2-2} holds for some $\underline j <j< \bar j$ . We also assume that $\Gamma_{\Lambda_{i+1}}(A,{\bf a_{i+1}}-(2(j+1),0),B)$ is nonempty
 since otherwise the inductive step is trivial. Then, we define the event:

$$F_j=\{\mbox{there is no open path from } A \mbox{ to } B \mbox{ in } \Lambda_{i+1} \mbox{ going through }C_j\}.$$
Since 
$$ \mu^{F_j^c}_{ \Lambda_{i+1}}(A,C_{j+1},B)\leq \mu_{ \Lambda_{i+1}}(A,C_{j},B)$$
and 
$\mu_{ \Lambda_{i+1}}(A,C_{j+1},B)$ is a convex combination of $\mu^{F_j^c}_{ \Lambda_{i+1}}(A,C_{j+1},B)$
and $\mu^{F_j}_{ \Lambda_{i+1}}(A,C_{j+1},B)$, 
$$ \mu_{ \Lambda_{i+1}}(A,C_{j+1},B)\leq \mu_{ \Lambda_{i}}(A,\mathbf a_{i+1},B)$$
will follow from the inductive hypothesis (in $j$) and
\begin{equation}\label{e2-3}
  \mu^{F_j}_{ \Lambda_{i+1}}(A,C_{j+1},B)\leq \mu_{\Lambda_i}(A,\mathbf a_{i+1},B).
\end{equation}
To prove \eqref{e2-3} we start noting that 
$$\mu^{F_j}_{ \Lambda_{i+1}}(A,C_{j+1},B)=\mu^{F_j}_{ \Lambda_{i+1}}(A,\mathbf a_{i+1}-(2(j+1),0),B). $$
We also claim that
 $$\mu^{F_j}_{ \Lambda_{i+1}}(A,\mathbf a_{i+1}-(2(j+1),0),B)\leq \mu_{ \Lambda_{i+1}}(A,\mathbf a_{i+1}-(2(j+1),0),B).$$
 This is proved using the same argument that was used for \eqref{e2}, except that we use a different part of the inductive hypothesis, namely, instead of using $\mu_{b(\tau_1,\tau_2)}(A,B) \leq \mu_{b(\tau_3,\tau_2)}(A,B)$, we now use $\mu_{b(\tau_1,\tau_3)}(A,B) \leq \mu_{b(\tau_1,\tau_2)}(A,B)$. 
Hence \eqref{e2-3} will follow from
\begin{equation}\label{e2-4}
  \mu_{ \Lambda_{i+1}}(A,\mathbf a_{i+1}-(2(j+1),0),B)\leq  \mu_{\Lambda_i}(A,\mathbf a_{i+1},B).
\end{equation}
But this last inequality can be proved in the same way 
we proved \eqref{e2-2} for $j=1$.
This completes the proof of \eqref{e1} for $i=0,\dots,n-2$. 

We now show that \eqref{e1} also holds for $i=n-1$.  Unlike what we did for $i<n-1$ we cannot use the inductive hypothesis to prove \eqref{e2} (see the step there that involves \eqref{e2prime} where we needed $i \leq n-2$). 
We may assume that $\mathbf a_n \in B$ and that $\Gamma_{\Lambda_{n-1}}(A,{\mathbf a_{n-1}}-(2,0,),{\bf a_n})$ is nonempty since otherwise \eqref{e1} is trivial for $i = n-1$.  
We let 
$$F=\{\mbox{there is no open path from } A \mbox{ to } B \mbox{ in } \Lambda_{n}\},$$
and note that
$\mu^{F^c}_{ \Lambda_{n-1}}(A,B)=\mu_{\Lambda_n}(A,B)$ and $$ \mu^{F}_{ \Lambda_{n-1}}(A,B)= \mu^F_{ \Lambda_{n-1}}(A,\mathbf a_{n-1}-(2,0),\mathbf a_n)\, .$$ Therefore \eqref{e1} for 
 $i=n-1$ will follow from:
\begin{equation}\label{e2-5}
\mu^F_{ \Lambda_{n-1}}(A,\mathbf a_{n-1}-(2,0),\mathbf a_n)\geq \mu_{ \Lambda_{n}}(A,B) \, .
\end{equation}
This inequality will follow from the following two inequalities:
\begin{equation}\label{e2-7}
 \mu^F_{ \Lambda_{n-1}}(A,\mathbf a_{n-1}-(2,0),\mathbf a_n)\geq  \mu_{ \Lambda_{n-1}}(A,\mathbf a_{n-1}-(2,0),\mathbf a_n),
\end{equation}
and 
\begin{equation}\label{e2-8}
 \mu_{ \Lambda_{n-1}}(A,\mathbf a_{n-1}-(2,0),\mathbf a_n)\geq  \mu_{ \Lambda_{n}}(A,B) \, .
\end{equation}
To prove \eqref{e2-7}, 
note that 
 \[
 \mu_{ \Lambda_{n-1}}(A,\mathbf a_{n-1}-(2,0),\mathbf a_n) = \mu_{\Lambda_{n-1}}(A,\mathbf a_{n-1} - (2,0)) \times 
 \mu(\mathbf a_{n-1}-(2,0),\mathbf a_n) \, ,
 \]
and since the event that there is an open path from $\mathbf a_{n-1} - (2,0)$ to $\mathbf a_n$ is independent of $F$, we also have
\[
 \mu^F_{ \Lambda_{n-1}}(A,\mathbf a_{n-1}-(2,0),\mathbf a_n) = \mu^F_{\Lambda_{n-1}}(A,\mathbf a_{n-1} - (2,0)) \times 
 \mu(\mathbf a_{n-1}-(2,0),\mathbf a_n) \, .
  \]
So by Lemma~\ref{l2}, \eqref{e2-7} is equivalent to the inequality
\[
\mu^F_{\Lambda_{n-1}}(A,\mathbf a_{n-1} - (2,0)) \geq \mu_{\Lambda_{n-1}}(A,\mathbf a_{n-1} - (2,0)) \, ,
\]
which follows from the inductive hypothesis and the argument that was used to prove the similar inequality \eqref{e2prime}.

It remains to prove \eqref{e2-8}.  This proof is very similar to the proof of \eqref{e2-1}.  Namely, we define the sets

$$C_j=\{\mathbf a_{n-1}-(2j,0),\dots, \mathbf a_{n-1}-(2,0)\},$$
where $j$ ranges from $\underline j $ to $\bar j$,  where  $\underline j $ is the lowest value of $j$ for which $\Gamma_{\Lambda_{n}}(A,C_{\underline j},B)$ is nonempty and $\bar j$ is
the largest value of $j$ for which $C_j$ is strictly to the right of $\tau_1$. 

Then $\mu_{\Lambda_n}(A,B) = \mu_{\Lambda_n}(A,C_{\bar j},B)$, so it is enough to prove by induction on $j$ that
\begin{equation} \label{eq99} \mu_{ \Lambda_{n-1}}(A,\mathbf a_{n-1}-(2,0),\mathbf a_n)\geq \mu_{\Lambda_n}(A,C_j,B)\end{equation} for each $j = \underline j,\dots,\bar j$. The case $j=\underline j$ is handled just as before, using the inductive hypothesis,
Lemma~\ref{l2} and Corollaries \ref{c3} and \ref{c4}.  The proof for the remaining values of $j$ is similar to but simpler than that part of the argument in the proof of \eqref{e2-1}.  It is simpler because the second marginal of the measure on the left of \eqref{eq99} is deterministic.  It is slightly different because the point $a_{n-1}-(2,0)$ is an element of the set $\Lambda_n$, but this difference does not cause any difficulties.

 \end{proof}

\bigskip

Now that Proposition~\ref{p1} is proved, we can use the corollaries to Lemma~\ref{l31} to improve the proposition by removing the restriction on the set $A$:
 \begin{corollary}\label{cp1}
Modify the hypotheses of Proposition~\ref{p1} as follows: remove  the assumption that the set $A$ strictly lie to the right of $\tau_3$  in the first part, and the assumption that $A$ lie strictly to the left of $\tau_3$ in the second part.  Then under this modified hypothesis, the conclusions of Proposition~\ref{p1} remain valid.
 \end{corollary}
 \begin{proof}{Corollary}{cp1}
We will show that $\mu_{b(\tau_1,\tau_3)}(A,B) \leq \mu_{b(\tau_1,\tau_2)}(A,B)$ without the assumption that $A$ is strictly to the left of $\tau_3$.  The proof of the other case is similar.  Let 
$$A_1=\{\mathbf a\in A: \mathbf a\mbox{ is strictly to the left of }\tau_3\} $$ and
$$A_2=A\setminus A_1.$$
Then,
$$ \mu_{b(\tau_1,\tau_3)}(A,B)=\mu_{b(\tau_1,\tau_3)}(A_1,B)\leq$$
$$\mu_{b(\tau_1,\tau_2)}(A_1,B)\leq \mu_{b(\tau_1,\tau_2)}(A_1\cup A_2,B)=$$
$$\mu_{b(\tau_1,\tau_2)}(A,B),$$
where the first inequality follows from Proposition~\ref{p1} and the second inequality 
follows from Corollary \ref{c3}.
 \end{proof}

\bigskip

We now proceed to prove Theorem \ref{t1}. We only prove the first of the four inequalities stated since the proof of each of the other three is similar.

\begin{proof}{Theorem}{t1}
Suppose $G$ is such that there exists at least one path from $A$ to $B$ in $\ell(G)$. Then $P_1(G)$ is bounded below. Now, let $\tau_G$ be the path provided by 
Lemma \ref{l1}. Then, 
$$\mu_{\ell(G)}(A,B)=\mu_{b(-\infty,\tau_G)}(A,B)\leq \mu_{b(-\infty,\infty)}(A,B)=\mu(A,B),$$
where the inequality follows from Corollary \ref{cp1}.
\end{proof}

\section{Alternative proof of Theorem \ref{t1}}
In this section we give another proof of Theorem  \ref{t1}, based on a Markov Chain introduced in \cite{BHK}. Assume  $A\subset L_0$ , $B\subset L_n$  and $E$ is a set of oriented edges in $\Lambda$ containing at least one path from $A$ to $B$. Now, let $S=\{0,1\}^E$. Each element $\eta$ of $S$ determines the state of the edges in $E$ in the natural way: $e\in E $ is open (closed )for $\eta$ if $\eta(e)=1$ ($\eta(e)=0$).  Now, we let $T$ be the subset of $S$ consisting of the elements for which there is an open path from $A$ to $B$. For $\eta \in T$, we let $\gamma_{\ell}(\eta)$ ($\gamma_{r}(\eta)$) be the leftmost (rightmost) open path from $A$ to $B$ under configuration $\eta$. We let $\sigma(S) $ be the 
$\sigma$-algebra of all subsets of $S$ and we let $P_p$ be the product  probability measure on $\sigma(S)$  whose marginals are Bernoulli with parameter $p$.  On the probability space $(S,\sigma(S),P_p)$,  we define for each $e\in E$ a random variable $X_e$ by means of $X_e(\eta)=\eta(e)$.

On the event $T$ we define $\Gamma_{\ell}$ ($\Gamma_r$) as the leftmost (rightmost ) open path from $A$ to $B$.
For a path $\gamma$ from $A$ to $B$ we let $\sigma_r(\gamma)$
be the $\sigma$-algebra  generated by  $\{X_e: e\in \gamma \mbox { or }  e\mbox{ is to the right of }\gamma\}$ and
we let $\sigma'_r(\gamma)$  be the $\sigma$-algebra  generated by  $\{X_e:   e\notin \gamma, e\mbox{ is to the right of }\gamma\}$ . Similarly we let $\sigma_{\ell}(\gamma)$
be the $\sigma$-algebra  generated by  $\{X_e: e\in \gamma \mbox { or }  e\mbox{ is to the left of }\gamma\}$ 
 and
we let $\sigma'_{\ell}(\gamma)$  be the $\sigma$-algebra  generated by  $\{X_e:   e\notin \gamma, e\mbox{ is to the left of }\gamma\}$.
We now note that the event $\{\Gamma_{\ell}=\gamma \}$ is $\sigma_{\ell}(\gamma)$-measurable. Therefore under the conditional
measure $P_p(\bullet\vert \Gamma_{\ell}=\gamma)$ the distribution of the state of the bonds which are strictly to the right of
$\gamma$ remains a Bernoulli  product measure of parameter $p$. Similarly,   under the conditional
measure $P_p(\bullet\vert \Gamma_r=\gamma)$ the distribution of the state of the bonds which are strictly to the left of
$\gamma$ remains a Bernoulli  product measure of parameter $p$. 
We now define a Markov Chain on T. Its  transition mechanism is given in two steps. For a given initial state $\eta_0$,
first we choose $\eta_{1/2}\in T$ by letting $\eta_{1/2}(e)=\eta_{0}(e)$ for all $e$ to the left of  $\gamma_{\ell}(\eta_0)$ or on 
$\gamma_{\ell}(\eta_0)$
and for the other elements of $E$ we let $\eta_{1/2}(e)$ be independent Bernoulli random variables with parameter $p$.
Once we have determined 
$\eta_{1/2}$ we let $\eta_1(e)=\eta_{1/2}(e)$ for all $e$  to the right of  $\gamma_r(\eta_{1/2})$ or on  $\gamma_r(\eta_{1/2})$
and for the other elements of $E$ we let $\eta_{1}(e)$ be independent Bernoulli random variables with parameter $p$ which are also 
independent of the random variables used to determine $\eta_{1/2}$. In the sequel we will need to consider this
Markov Chain for different sets $E$. We will call it the Markov Chain associated to $E$.  
In the sequel, we extend the notation of the previous sections by letting 
$\Gamma_E(A,B)$ be the set of paths from $A$ to $B$ whose edges are in $E$.


\begin{proposition}\label{p3}(van den Berg, H\"aggstr\"om, Kahn)
 The measure $P_p(\bullet \vert T)$ is invariant for the Markov chain.
\end{proposition}
\begin{proof}{Proposition}{p3}
We show that if the initial state of the chain $\eta_0$ is chosen according to the distribution $P_p(\bullet \vert T)$,
then $\eta_{1/2}$ has the same distribution. A similar argument will then show that $\eta_1$ has the same distribution as
 $\eta_{1/2}$. 
Let $\gamma$ be an arbitrary path in $\Gamma_E(A,B)$ and let
$$S_{\gamma}=\{\eta\in T: \gamma \mbox{ is the leftmost open path from }A \mbox{ to }B$$
$$ \mbox{ under configuration }s\} .$$
Then $(S_{\gamma}: \gamma \in \Gamma_E(A,B)$ is a partition of $T$ and  $P_p(\bullet \vert T)$ is a convex combination of
the measures $(P_p(\bullet \vert S_{\gamma}))_{\gamma \in \Gamma_E(A,B)}$.
Therefore, it suffices to show that if $\eta_0$ is distributed according to some $P_p(\bullet \vert S_{\gamma})$ then 
$\eta_{1/2}$ is also distributed according to that measure. But this is an immediate consequence of the way we obtain $\eta_{1/2}$ from $\eta_0$
and the already observed fact that the under the conditional
measure $P_p(\bullet\vert \Gamma_{\ell}=\gamma)$ the distribution of the state of the bonds which are strictly to the right of
$\gamma$ remains a Bernoulli  product measure of parameter $p$..

\end{proof}

Since the Markov Chain associated to
 $E$ is obviously irreducible and aperiodic we deduce from this proposition the following:
 
 \begin{corollary}\label{c2}
 From any initial distribution, the Markov Chain associated to $E$ converges to  $P_p(\bullet \vert T)$
  \end{corollary}
\begin{proof}{Theorem}{t1}
As before, we only prove the first inequality, since the other proofs are similar.
Let $E$ be the set of edges belonging to paths in $\Gamma(A,B)$  and let $E_{\ell(G)}$ be the set of edges belonging to paths in $\Gamma_{\ell(G)}(A,B)$. We let  $T_1$ be the subset of elements of $\{0,1\}^E$ for which there exists an open path from 
$A$ to $B$ and we $T_2$ be the subset of elements of 
 $\{0,1\}^{E_{\ell(G)}}$ for which there exits an an open path from 
$A$ to $B$. We now construct a Markov Chain in 
$$X=\{(\eta,\xi)\in T_1\times T_2: \gamma_{\ell}(\eta) \geq \gamma_{\ell}(\xi),  \gamma_{r}(\eta) \geq \gamma_{r}(\xi)\}$$
 whose first and second marginals are as the Markov Chains associated to $E$ and to $E_{\ell(G)}$ respectively.  This is done as follows: 
Assume $(\eta_0\xi_0)\in X$, then  let $\{Y_e: e\mbox{ strictly to the right of } \gamma_{\ell}(\xi)\}$ be a collection of i.i.d Bernoulli random variables of parameter $p$. First note by definition of $X$,
 $\gamma_{\ell}(\eta) \geq \gamma_{\ell}(\xi)$, then define $\eta_{1/2}$ and $\xi_{1/2}$ as follows:\newline
$\eta_{1/2}(e)= \eta_0(e)$ for all $e$ on $\gamma_{\ell}(\eta_0)$ or to the left of $\gamma_{\ell}(\eta_0)$,  \newline
$\eta_{1/2}(e)=Y_e$ for all $e$ strictly to the right of $\gamma_{\ell}(\eta_0)$, \newline
$\xi_{1/2}(e)= \xi_0(e)$ for all $e$ on $\gamma_{\ell}(\xi_0)$ or to the left of $\gamma_{\ell}(\xi_0)$,  \newline
$\xi_{1/2}(e)=Y_e$ for all $e$ strictly to the right of $\gamma_{\ell}(\xi_0)$. \newline
After that note that  $(\eta_{1/2},\xi_{1/2})\in X$ and let $ \{Z_e: e\mbox{ strictly to the left of } \gamma_{r}(\eta)\}$ be
 collection of i.i.d Bernoulli random variables of parameter $p$ which is independent of the random variables $X_e$.
 Finally, define $\eta_{1}$ and $\xi_{1}$ as follows:\newline
 $\eta_{1}(e)= \eta_{1/2}(e)$ for all $e$ on $\gamma_{r}(\eta_{1/2})$ or to the right of $\gamma_{r}(\eta_{1/2})$,  \newline
$\eta_{1}(e)=Z_e$ for all $e$ strictly to the left of $\gamma_{r}(\eta_{1/2})$, \newline
$\xi_{1}(e)= \xi_{1/2}(e)$ for all $e$ on $\gamma_{r}(\xi_{1/2})$ or to the right of $\gamma_{r}(\xi_{1/2})$,  \newline
$\xi_{1}(e)=Z_e$ for all $e$ strictly to the left of $\gamma_{r}(\xi_{1/2})$. \newline
We can now complete our proof: let $\Phi$ be a bounded  increasing function on $\Gamma(A,B)$ and let $(\eta_0,\xi_0)$
be an element of $X$. Then $(\eta_n,\xi_n)\in X$ a.s. $\forall n$. Therefore, 
$$\Phi(\gamma_{\ell }(\eta_n))\geq \Phi(\gamma_{\ell }(\xi_n)) a.s.  \ \  \forall n.$$ Hence,
$$E(\Phi(\gamma_{\ell }(\eta_n)))\geq E(\Phi(\gamma_{\ell }(\xi_n)))\ \ \forall n,$$
and applying Corollary \ref{c2} to both sides of the inequality we get $\mu(A,B)\geq \mu_{\ell(G)}(A,B)$.

\end{proof}

\section{Generalizations and extensions of Theorem \ref{t1}} \label{sec5}
We first discuss generalizations of the oriented bond percolation model that do not require any change in our proofs of the main results, except for one case in which the Markov chain proof does not seem to work.  Then we consider oriented site percolation and the contact process.

It is easy to check that we never made any use of the assumption that all of the bonds have the same probability of being open.  In fact, we could assign a different probability to each bond, and the proofs will continue to work without any changes.  It may seem that this is an uninteresting generalization, but we will see that it turns out to be relevant when we use percolation models to approximate the continuous time contact process.

Another easy generalization involves the assumption of independence between bonds.  Nowhere in our original proofs of the main result do we need the openness of two bonds to be independent if those two bonds emerge from the same site.  That is, if $(x,y)$ is a site in $\Lambda$, then the events that the two bonds that connect $(x,y)$ to $(x\pm 1,y+1)$ are open can be correlated arbitrarily.  Independence is only needed for bonds that emerge from different sites. See the comment that is found in the proof of Proposition~\ref{p1}, at the end of the paragraph where the random set $\Phi$ is defined. This is one situation where the Markov chain proof does not work quite so well; it does not seem to be valid in the case of negative correlations.  

We now turn to oriented site percolation.  Let 
\[
\overline{\Lambda} = \{(x,y): x,y \in \Z, y \geq 0\} \, .
\]
Fix integers $a \leq 0 < b$, and for each $(x,y) \in \overline{\Lambda}$, introduce oriented bonds from $(x,y)$ to $(x+k,y+1)$ for $a \leq k \leq b$.  
All of the bonds are open.  The sites are open independently of each other, with probability $p \in (0,1)$ (we could also allow different probabilities for different sites).  Paths are defined in the obvious way, and open paths are paths in which all of the sites are open.  

The case where $a=0$ and $ b=1$ is equivalent to the standard oriented site percolation in $\Z^2$, where there are two bonds per site.  Not surprisingly, our proofs of the main result can be modified in a routine way to cover this case.  The only real difference involves the places in the proofs of Lemma~\ref{l31} and Proposition~\ref{p1} where a random set $\Phi$ is defined.  In the case of oriented bond percolation, $\Phi$ was defined in each proof to be a certain random set of open edges belonging to paths that either started in some set $A$ or ended in a set $B$. For the oriented site percolation model, it is best to define $\Phi$ to be an analogous random set of open sites, and then enlarging $\Phi$ to include all of the (typically closed) sites that are at the ends of bonds that are connected to sites in $\Phi$.  Then the proofs can be continued as before, looking at open paths that are strictly to the right of the enlarged version of $\Phi$.  

Once the case $a=0$, $ b=1$ is handled, it is quite routine to further modify the proof to cover arbitrary $a \leq 0 < b$, which is to say that our main result holds for finite range oriented site percolation in 2 dimensions.  This highlights a significant difference between oriented bond percolation and oriented site percolation.  In oriented bond percolation, the leftmost and rightmost paths may not even exist when there is the possibility that bonds cross one another, as will be the case when $b-a > 1$.  But in oriented site percolation, leftmost and rightmost paths always exist, for any choice of $a,b$.  

Finally, we briefly discuss the contact process.  By treating the $y$-coordinate in $\Lambda$ as the time variable, one can obtain various versions of the discrete time contact process from oriented percolation.  The standard model is equivalent to oriented site percolation.  Variations on this model can be obtained from oriented bond percolation, and also by looking at mixed models in which both the sites and the bonds can be open or closed.  Not surprisingly, our main results applies to many mixed percolation models, and hence to many different discrete time contact processes.

One way to extend our results to the continuous time contact process is to approximate continuous time with discrete time. This method works easiest for the one-sided nearest neighbor contact process.  Then we could use the oriented bond percolation model that is the setting for most of this paper, but it is perhaps more natural to do oriented bond percolation on an equivalent graph: the set of sites is $\overline{\Lambda}$ and the oriented bonds are those that correspond to $a=0$ and $b=1$.  That is, two oriented bonds emerge from each site $(x,y)$, a ``vertical'' bond connecting it to the site
$(x,y+1)$ and a ``contact'' bond connecting it to $(x+1,y+1)$.  

With this setup, we can approximate the continuous time one-sided contact process by letting  the contact bonds be open with small probability $h > 0$ and letting the   vertical bonds be open with probability $1- \varepsilon h$, where $\varepsilon \geq 0$ is a parameter of the model.  Then letting $h \downarrow 0$ and rescaling time by a factor of $h$ produces the continuous-time model.  This shows why it can be desirable to allow the bonds to have different probabilities of being open.

One can use a similar approximation method for the two-sided nearest neighbor contact process.  In this case, we use the graph $\Lambda$ that is the setting for the bulk of this paper, but we add additional oriented ``vertical bonds'' that connect each site $(x,y) \in \Lambda$ to the site $(x,y+2)$. Then by assigning appropriate probabilities to the bonds (different for the vertical bonds than for the diagonal bonds), one obtains a percolation model that depends on a parameter $h$, and this model converges to the two-sided nearest neighbor contact process as $h\to 0$.
See  \cite{sued} for further details.  The bottom line is that the results in this paper all apply to one the one-sided and two-sided contact processes in continuous time.

\section{Applications of Theorem \ref{t1}} \label{applications}
One reason for our interest in extreme paths is that they provide us with a useful way to analyze various conditional probabilities, and with the help of Theorem \ref{t1}, we are able to make comparisons that go beyond the usual correlation inequalities that are familiar in percolation theory. The results in this section apply to the more general models discussed in the previous section, except that Corollary~\ref{corl2} requires translation invariance, so that all of the bond (or site) probabilities must be the same.

We will rely on a key fact about extreme paths.  It is that if $\gamma$ is a path, the event that $\gamma$ is a rightmost (leftmost) extremal path is measurable with respect to the states of the bonds that are to the right (left) of $\gamma$ and hence this event is independent of the states of the bonds that are strictly to the left (right) of $\gamma$. (For more details, see the discussion and definitions at the beginning of the previous section.)

The following result and its proof show how we use this fact in conjunction with Theorem \ref{t1} to compare several different conditional probabilities.

\begin{lemma} \label{lembasic}
Let $n > 0$, let $A$ be a finite subset of $L_0$ and let $B_1,B_2,B_3 $ be finite subsets of $L_n$.  Suppose that $B_1$ is strictly to the left of $B_2$ and that $B_2$ is strictly to the left of $B_3$.  For $i=1,2,3$, define the events
\[
H_i = \{\mbox{there exists an open path from $A$ to $B_i$}\}
\]
If the events $H_2$ and $H_3$ have nonzero probability, then
\begin{equation} \label{basicineq}
P(H_1 \mid H_2 \cap H_3^c) \geq P(H_1 \mid H_2)\geq P(H_1 \mid H_2 \cap H_3 )  
\end{equation}
and
\begin{equation} \label{basicineq2}
P(H_1 \mid H_2) \geq P(H_1 \mid H_3) \, .
\end{equation}
Furthermore, if $A$ consists of a single site $(x,0)$, then $P(H_1 \mid H_2)$ is nonincreasing in $x$ for all $x$ such that the event $H_2$ has positive probability.  
\end{lemma}

\begin{proof}{Lemma}{lembasic}
We begin by proving the following inequality:
\begin{equation} \label{con1}
\nu^{H_3^c}(A,B_2) \leq \nu(A,B_2) \, .
\end{equation}
The proof of \eqref{con1} is similar to arguments that we have made before. The event $H_3^c$ is the disjoint union of events of the form $\{\Phi = \varphi\}$, where $\Phi$ is the random set of all edges that are contained in open paths that end in $B_3$.  The left side of \eqref{con1} is a convex combination of the measures $\nu^{\{\Phi = \varphi\}}(A,B_2)$, and for each $\varphi$,  $\nu^{\{\Phi = \varphi\}}(A,B_2)=\nu_{\ell(G)}(A,B_2)$, where $G$ is the set that contains the endpoints of the edges in $\varphi$.  The inequality in \eqref{con1} now follows from Theorem \ref{t1}.

We now use \eqref{con1} to prove the first inequality in \eqref{basicineq}.  Let $\Gamma$ be the rightmost open path from $A$ to $B_2$, assuming that such a path exists, which is the same as assuming that $H_2$ occurs.
The left side of \eqref{con1} is the conditional distribution of $\Gamma$ given $H_2 \cap H_3^c$ and the right side of \eqref{con1} is the conditional distribution of $\Gamma$ given $H_2$.  In either case, the event $H_1$ occurs if and only if there is an open path from $A \cup \Gamma$ to $B_1$.

Let $\gamma$ be any path from $A$ to $B_1$.  Given the event $H_2 \cap \{\Gamma = \gamma\}$, the edges strictly to the left of $\gamma$ are each open with probability $p$ and they are independent of each other.  This statement about the edges left of $\gamma$ also holds true given the event $H_2 \cap H_3^c \cap \{\Gamma = \gamma \} $ because of the assumption that $H_3$ is strictly to the right of $H_2$.  In either case, whether or not there is an open path from $A \cup \gamma$ to $B_1$ is determined in the same way by the openness of the edges that are strictly to the left of $\gamma$.  Thus, there is a function $\varphi$ on the set of all paths $\gamma$ from $A$ to $B_2$ such that
\[
\varphi(\gamma) = P(H_1 \mid H_2 \cap \{\Gamma = \gamma\}) = P(H_1 \mid H_2 \cap H_3^c \cap \{\Gamma = \gamma \}) \, ,
\]
and we have
\[
E(\varphi(\Gamma) \mid H_2) = P(H_1 \mid H_2) \quad \mbox{and} \quad E(\varphi(\Gamma)\mid H_2 \cap H_3^c) = P(H_1 \mid H_2 \cap H_3^c) \, .
\]
Clearly, $\varphi$ is monotone with respect to the partial ordering on paths, so
The first inequality in \eqref{basicineq} now follows from \eqref{con1}.  The second inequality in \eqref{basicineq} follows immediately from the first inequality and the fact that the middle expression in \eqref{con1} is a convex combination of the first and third expressions.

To prove \eqref{basicineq2}, we note that repeated applications of Corollary \ref{t2} imply that
\[
\nu(A,B_2) \leq \nu(A,B_3) \, .
\]
(See the proof of Corollary \ref{c3} for a similar argument involving repeated applications of Corollary \ref{t2}.)
Now \eqref{basicineq2} follows, in the same way that \eqref{basicineq} followed from \eqref{con1}.

To prove the last part of the lemma, we note that by two applications of Corollary \ref{t2},
\[
\nu((x,0),B_2) \leq \nu((x+2,0),B_2)
\]
for all $x$ such that there exist paths from $(x,0)$ and $(x+2,m)$ to $B_2$.  The last part of the lemma is now proved in the same way that \eqref{basicineq} and \eqref{basicineq2} were proved.
\end{proof}

The first inequality in \eqref{basicineq} may seem counterintuitive.  We know that the occurence of a ``negative'' event like $H_e^c$ makes a ``positive'' event like $H_1$ less likely to occur.  But the first inequality says informally that once $H_2$ occurs, the additional occurrence of the negative event  $H_3^c$ makes $H_1$ {\em more likely} to occur. Here is another way to state this surprising result:

\begin{corollary} \label{corl1}
Let $H_1,H_2,H_3$ be as in Lemma \ref{lembasic}.  Then given $H_2$, the events $H_1$ and $H_3$ are conditionally negatively correlated.  
\end{corollary}
\begin{proof}{Corollary}{corl1}
The proof is elementary, using the second inequality in Lemma \ref{lembasic}:

$$
P(H_1 \cap H_3 \mid H_2) =  \frac{P(H_1 \cap H_2 \cap H_3)}{P(H_2)} = \frac{P(H_1 \cap H_2 \cap H_3)}{P(H_2 \cap H_3)}P(H_3 \mid H_2)$$ $$ = P(H_1 \mid H_2 \cap H_3)P(H_3 \mid H_2) \leq P(H_1 \mid H_2)P(H_3 \mid H_2) \, .$$
\end{proof}

Here is another application of Lemma \ref{lembasic}.  It is a rather natural monotonicity involving certain percolation probabilities. It is somewhat surprising that its proof seems to require the consequences of something as sophisticated as Theorem \ref{t1}. We note that this result clearly depends on some translation invariance, so it requires all of the bond probabilities (or site probabilities in the case of oriented site percolation) to be the same. A  different proof of this result is given in \cite{sued}, but we believe the one given here is more natural and easier to follow.

\begin{corollary} \label{corl2}
Let $0 \leq m < n$ and let $x,y$ be integers such that $(x,m)\in L_m$ and $(y,n)\in L_n$.  Let $A_{x,y}$ be the event that there is an open path from $(x,m)$ to $(y,n)$.  Then $P(A_{x,y})$ is nonincreasing in $|x-y|$.  
\end{corollary}

\begin{proof}{Corollary}{corl2}
Because of the natural symmetries built into the percolation model, we may assume without loss of generality that $m=0$, $x=0$ and $y \geq 0$.   The obvious inductive argument reduces the proof to showing that
\begin{equation} \label{condineq1}
P(A_{0,y}) \geq P(A_{0,y+2}) \, ,
\end{equation} 
where we may assume that $y$ is such that there exists at least one path from $(0,0)$ to $(n,y+2)$.  Since $y \geq 0$, this assumption implies that there also exists at least one path from $(0,0)$ to $(n,y)$.  

Under these circumstances,
to prove \eqref{condineq1}, it is enough to prove that
\begin{equation} \label{condineq2}
P(A_{0,y} \mid A_{0,y+2}) \geq P(A_{0,y+2}\mid A_{0,y})
\end{equation}
since the numerators in the expressions for the two conditional probabilities in \eqref{condineq2} are the same and since the denominators in these expressions are the two sides of \eqref{condineq1} (in reverse order).

To prove \eqref{condineq2}, we first use left-right symmetry and then translation invariance to get
\[
P(A_{0,y+2} \mid A_{0,y}) = P(A_{0,-y-2} \mid A_{0,-y}) = P(A_{2y+2,\, y} \mid A_{2y+2,\, y+2}) \, .
\]
The last part of Lemma \ref{lembasic} implies that we make the right side of this equation no smaller if we replace $2y+2$ by $2y$, then by $2y-2$, then by $2y-4$, and so on.  Since we assumed that $y \geq 0$, this gives us
\[
P(A_{2y+2,\, y} \mid A_{2y+2,\, y+2}) \leq P(A_{0,y} \mid A_{0,y+2})
\]
proving \eqref{condineq2}, and thus \eqref{condineq1}.  Note that our assumptions about $y$ ensure that all of the relevant events  in these applications of Lemma \ref{lembasic} have positive probability, as required by the hypotheses of that lemma.
\end{proof}

\end{document}